\documentclass[12pt,A4]{amsart}
\usepackage{amsfonts,amsthm,amsmath,latexsym, amscd, epsfig, epic,color}
\usepackage{amssymb}
\usepackage{rotating}
\usepackage{tikz}
\usepackage{verbatim}
\usepackage{amscd}
\usepackage[latin2]{inputenc}
\usepackage{t1enc}
\usepackage[mathscr]{eucal}
\usepackage{indentfirst}
\usepackage{graphicx}
\usepackage{graphics}
\usepackage{pict2e}
\usepackage{mathrsfs}
\usepackage{enumerate}
\usepackage[pagebackref]{hyperref}
\usepackage{amsmath, amssymb, amsthm}
\usepackage{tikz}
\usepackage{mathtools}
\usepackage{bm}
\usepackage{url}
\usepackage{framed}

\hypersetup{backref, pagebackref, colorlinks=true}
\usepackage{cite}
\usepackage{color}
\usepackage{epic}
\usepackage{hyperref} 
\usepackage{tkz-graph}

\numberwithin{equation}{section}
\def\red{\textcolor{red}}

\theoremstyle{plain}
\newtheorem{theorem}{Theorem}[section]
\newtheorem{lemma}[theorem]{Lemma}

\newtheorem{proposition}[theorem]{Proposition}

\theoremstyle{definition}

\newtheorem{example}[theorem]{Example}
\newtheorem{conjecture}[theorem]{Conjecture}

\newtheorem{?}[theorem]{Problem}

\newcommand{\N}{\mathbb{N}}

\allowdisplaybreaks

\def\S{\mathfrak{S}}

\def\A{\mathcal{A}}
\def\PA{\mathcal{PA}}

\def\last{\mathrm{last}}

\def\zero{\mathrm{zero}}

\def\len{\mathrm{len}}

\def\asc{\mathrm{asc}}

\def\wt{\operatorname{wt}}

\def\I{\operatorname{{\bf I}}}

\def\rmin{\operatorname{rmin}}
\def\lmin{\operatorname{lmin}}
\def\rmax{\operatorname{rmax}}
\def\lmax{\operatorname{lmax}}
\def\maxi{\operatorname{max}}

\def\rep{\operatorname{rep}}

\def\tt{\mathbf{t}}

\def\boxit#1{\leavevmode\hbox{\vrule\vtop{\vbox{\kern.33333pt\hrule\kern1pt\hbox{\kern1pt\vbox{#1}\kern1pt}}\kern1pt\hrule}\vrule}}

\usepackage{collectbox}


\begin{document}

\title[Patterns in inversion sequences]{On $\underline{12}0$-avoiding inversion and ascent sequences}

\author[Z. Lin]{Zhicong Lin}
\address[Zhicong Lin]{Research Center for Mathematics and Interdisciplinary Sciences, Shandong University, Qingdao 266237, P.R. China}
\email{linz@sdu.edu.cn}

\author[S. Fu]{Shishuo Fu}
\address[Shishuo Fu]{College of Mathematics and Statistics, Chongqing University, Huxi campus, Chongqing 401331, P.R. China}
\email{fsshuo@cqu.edu.cn}


\date{\today}

\begin{abstract}
Recently, Yan and the first named author investigated systematically the enumeration of inversion or ascent sequences avoiding vincular patterns of length $3$, where two of the three letters are required to be adjacent. They established many connections with familiar combinatorial families and proposed several interesting conjectures. The objective of this paper is to address two of their conjectures concerning the enumeration of $\underline{12}0$-avoiding inversion or ascent sequences.
\end{abstract}

\keywords{Vincular patterns; Inversion/Ascent sequences; triangular binomial coefficients; Powered Catalan numbers}

\maketitle
\section{Introduction}
Since the publications of Duncan and Steingr\'imsson~\cite{ds}, Corteel, Martinez, Savage and Weselcouch~\cite{cor} and Mansour and Shattuck~\cite{mash}, there has been increasing interest in counting pattern avoiding ascent/inversion sequences~\cite{auli,auli2,bp,bbgr,bgrr,cjl,cddds,kl,kl2,lin,lin2,ly,ms,yan0,yan}. In particular, motivated by the study of generalized patterns in permutations~\cite{bas,cla}, Yan and the first named author~\cite{ly} carried out the systematic study of ascent/inversion sequences avoiding vincular patterns of length $3$. They reported many nice connections with familiar combinatorial families and posed several challenging enumeration conjectures. The objective of this paper is to address two of their conjectures concerning the pattern $\underline{12}0$ in  ascent/inversion sequences. It turns out that  $\underline{12}0$-avoiding ascent and inversion sequences possess attractive enumeration results albeit having elusive structure.

Before stating our results, we need to review some definitions on inversion sequences. An integer sequence $e=e_1e_2\ldots e_n$ of length $n$ is an {\em inversion sequence} if $0\leq e_i<i$ for all $1\leq i\leq n$. Inversion sequences of length $n$, denoted $\I_n$, are in natural bijection with permutations $\S_n$ of $[n]:=\{1,2,\ldots,n\}$ via the famous {\em Lehmer code} (see~\cite{cor,kl}). The set $\A_n$ of {\em ascent sequences} of length $n$ consists of $e\in\I_n$ such that 
$$
e_{i+1}\leq \asc(e_1e_2\ldots e_{i})+1
$$
for all $1\leq i<n$, where $\asc(e_1e_2\ldots e_i):=|\{\ell\in[i-1]:e_{\ell}<e_{\ell+1}\}|$ is the number of {\em ascents} of $e_1e_2\ldots e_i$.
As one of the most important subsets of inversion sequences, ascent sequences were introduced by Bousquet-M\'elou, Claesson, Dukes and Kitaev~\cite{bcdk} to encode the $(2$+$2)$-free posets. Many remarkable connections between inversion (resp.~ascent) sequences and permutations (resp.~restricted permutations) with unexpected applications have been found in the literature; see~\cite{sav,kl,fjlyz} and the references therein. 

Permutations and inversion sequences can both be viewed as words on $\N$. A word $w=w_1w_2\cdots w_n$ contains  a pattern $p=p_1p_2\cdots p_k$ if there exists $i_1<i_2<\cdots<i_k$ such that the subword $w_{i_1}w_{i_2}\cdots w_{i_k}$ of $w$ is order isomorphic  to $p$. In addition, if some consecutive letters in a pattern $p$ are underlined, then we further require that in any occurrence of $p$, the letters corresponding to these underlined letters be adjacent in $w$. Such generalized patterns are known as {\em vincular patterns} (cf.~Kitaev's book~\cite[pp.2]{ki}), which were introduced in the classification of Mahonian statistics by Babson and Steingr\'imsson~\cite{bs}. If a word $w$ does not contain an occurrence of  a vincular pattern $p$, then $w$ is said to {\em avoid} the pattern $p$. For example, an inversion sequence $e\in\I_n$ is  $\underline{12}0$-avoiding if there does not exist indices $i$ and $j$, $2\leq i<j\leq n$, such that $e_j<e_{i-1}<e_i$.
For a set $W$ of words, the set of $p$-avoiding words in $W$ is denoted by $W(p)$.

For $n\geq1$ and $0\leq k\leq n-1$, let 
$$
T(n,k):={{k+2\choose2}+n-k-2\choose n-k-1}.
$$
The triangle $\{T(n,k)\}_{0\leq k\leq n-1}^{n\geq1}$ is known as the triangle of {\em triangular binomial coefficients} and appears as A098568 in the OEIS~\cite{oeis}. Lin and Yan~\cite{ly} proved that $T(n,k)$ enumerates ascent sequences $e\in\A_n(1\underline{01})$ with $\asc(e)=k$ and 
 conjectured the following different interpretation.

\begin{conjecture}[Lin and Yan~\text{\cite[Conj.~3.4]{ly}}]\label{asc:12-0}
For $n\geq1$ and $0\leq k\leq n-1$, we have
$$
|\{e\in\A_n(\underline{12}0):\asc(e)=k\}|=T(n,k).
$$
\end{conjecture}

Following Yan~\cite{yan0}, an ascent sequence $e\in\A_n$ is said to be {\em primitive} if $e_i\neq e_{i+1}$ for all $i\in[n-1]$. Let $\PA_n$ be the set of primitive ascent sequences of length $n$. It was observed in~\cite{ly} that Conjecture~\ref{asc:12-0} is equivalent to 
\begin{equation}\label{PA:12-0}
|\{e\in\PA_n(\underline{12}0):\asc(e)=k\}|=\binom{\binom{k+1}{2}}{n-k-1}.
\end{equation}
Alternatively, it suffices to establish the generating function formula
\begin{equation}\label{aztec}
\sum_{e}x^{\len(e)-k-1}=(1+x)^{{k+1\choose2}}, 
\end{equation}
where the sum runs over all $\underline{12}0$-avoiding primitive inversion sequences with $k$ ascents and $\len(e)$ is the length of $e$. Interestingly,  the right-hand side of~\eqref{aztec} is also the generating function for domino tilings of Azetec diamond of order $k$ by the number of  horizontal dominoes, a celebrated  result of Elkies, Kuperberg, Larsen and Propp~\cite{eklp}.

Another conjecture in~\cite{ly} concerns an interpretation for the refined powered Catalan numbers in terms of $\underline{12}0$-avoiding inversion sequences.
The integer sequence of {\em powered Catalan numbers} $\{p_n\}_{n\geq1}$ is registered on~\cite{oeis} as A113227, whose first few terms are
$$
1,2,6,23,105,549,3207,20577,143239,\ldots.
$$
It is known that the pattern avoiding classes
$\S_n(1\underline{23}4)$, $\S_n(1\underline{32}4)$, $\S_n(1\underline{34}2)$, $\S_n(1\underline{43}2)$, $\I_n(101)$ and $\I_n(110)$ are all counted by the powered  Catalan number $p_n$ (see~\cite{bbgr,cor,ms} and the references cited therein).
The number $p_n$ has a natural refinement by $p_n=\sum_{k=1}^n c_{n,k}$, where $c_{n,k}$ are defined recursively by
\begin{align}\label{recurrence}
\begin{cases}
c_{1,1}=1\text{ and }c_{n,0}=0, \text{ for } n\ge 1,\\
c_{n,k}=c_{n-1,k-1}+k\sum_{j=k}^{n-1}c_{n-1,j}, \text{ for $n\geq2$ and $1\leq k\leq n$}.
\end{cases}
\end{align}
Corteel, Martinez,  Savage and  Weselcouch~\cite{cor} proved that the cardinality of $\I_n(101)$ or $\I_n(110)$ is $p_n$ by showing
\begin{align}\label{101=110 with zero}
|\{e\in\I_n(101):\zero(e)=k\}|=|\{e\in\I_n(110):\zero(e)=k\}|=c_{n,k},
\end{align}
where $\zero(e)$ is the number of zero entries of $e$. Lin and Yan~\cite{ly} showed that $\I_n(\underline{12}0)$ has cardinality $p_n$ by establishing a bijection between $\S_n(3\underline{21}4)$ and $\I_n(\underline{12}0)$ but were unable to prove the following refinement.

\begin{conjecture}[\text{Lin and Yan~\cite[Conj.~2.20]{ly}}]\label{inv:12-0}
For $n\ge 1$ and $1\leq k\leq n$, we have
$$
|\{e\in\I_n(\underline{12}0):\zero(e)=k\}|=c_{n,k}.
$$
\end{conjecture}

In this paper, we confirm the above two conjectures. 

The rest of this paper is organized as follows. In Section~\ref{sec:2}, we prove Conjecture~\ref{asc:12-0} by considering the last entries of $\underline{12}0$-avoiding ascent sequences. In Section~\ref{sec:3}, we prove Conjecture~\ref{inv:12-0} via a well-designed algorithm for constructing  $\underline{12}0$-avoiding inversion sequences. We will also consider the last entry statistic of $\underline{12}0$-avoiding inversion sequences, which leads to a new succession rule for the powered Catalan numbers.  Finally, we end this paper with two tempting  equidistribution conjectures concerning the open problem to enumerate  $\underline{23}14$-avoiding permutations.

\section{On $\underline{12}0$-avoiding ascent sequences}
\label{sec:2}
This section is devoted to the proof of Conjecture~\ref{asc:12-0}. We begin with  a different characterization of $\PA_n(\underline{12}0)$, which is more convenient for our enumerative purpose.

For a given $e\in\PA_n$, since it is primitive, each consecutive pair $(e_i,e_{i+1})$ forms either a descent, or an ascent. Now we can uniquely partition $e$ with ``$\slash$'', into maximal decreasing subsequences called {\em runs}. Let $\tt(e)$ be the subsequence formed by the least entry in each run of $e$, and we call it the {\em tail sequence} of $e$. For example, $$e=0102324325=0\slash10\slash2\slash32\slash432\slash5\quad\text{and}\quad\tt(e)=002225.$$
We have the following characterization of $\PA_n(\underline{12}0)$ using tail sequences.

\begin{lemma}\label{char:12-0}
For  any $e\in\PA_n$ with $\asc(e)=k$, we have
$$e\in\PA_n(\underline{12}0) \text{ if and only if $\tt(e)\in\I_{k+1}$ is non-decreasing.} $$
\end{lemma}
\begin{proof}
Clearly, $\tt(e)\in\I_{k+1}$ is a consequence  of $\asc(e)=k$ and the definitions of primitive ascent sequences and tail sequences. Now we show that $e$ contains a $\underline{12}0$ pattern if and only if $\tt(e)=t_1t_2\ldots t_{k+1}$ contains a descent.

Suppose the triple $e_ie_{i+1}e_j$ forms a $\underline{12}0$ pattern in $e$, then $e_i$ must be the tail of a run. Suppose $e_l$ is the tail of the run that contains $e_j$, for some $l\ge j$. We see $e_i>e_j\ge e_l$, hence $\tt(e)$, containing $e_i$ and $e_l$, must have a descent. Conversely, suppose $t_i>t_{i+1}$ is a descent in $\tt(e)$, and suppose the tails of the $i$-th and $(i+1)$-th runs, and the largest entry in the $(i+1)$-th run in the original sequence $e$, are $e_p$ ($=t_i$), $e_{q}$ ($=t_{i+1}$) and $e_{p+1}$ respectively. Then $e_pe_{p+1}e_q$ forms a $\underline{12}0$ pattern in $e$, which completes the proof of the lemma.
\end{proof}


Let $\PA$ denotes the set of all primitive ascent sequences. For each $e\in\PA$, define the weight of $e$ by $\wt(e) =\prod\limits_{i=1}^{\len(e)}\wt(e_i)$, where 
\begin{align*}
\wt(e_i)&:=\begin{cases}
1 & \text{if $e_i$ is a tail,}\\
x & \text{otherwise.}
\end{cases}
\end{align*}
If $e$ has $k$ ascents, then it has exactly $k+1$ tails hence $\wt(e)=x^{\len(e)-k-1}$.
Therefore, Eq.~\eqref{aztec} can be rewritten as
\begin{equation}\label{sum over PA}
 \sum_{{e\in\PA(\underline{12}0)}\atop{\asc(e)=k}}\wt(e)=(1+x)^{{k+1\choose 2}},
\end{equation}
which is equivalent to Conjecture~\ref{asc:12-0}.
In order to prove~\eqref{sum over PA}, we introduce the refined enumerator 
$$
f_{k,i}(x):=\sum_{e\in\PA_{k,i}(\underline{12}0)}\wt(e), 
$$
where $\PA_{k,i}(\underline{12}0)$ is the set of all $e\in\PA(\underline{12}0)$ with $\asc(e)=k$ and the last entry of $e$ being $i$. We have the following recursion for $f_{k,i}(x)$.
\begin{lemma}For $k\geq0$ and $0\leq i\leq k+1$, we have the recursion
\begin{equation}\label{rec:12-0}
f_{k+1,i}(x)=\sum_{j=0}^{i}(1+x)^{k+1-i}f_{k,j}(x)-f_{k,i}(x)
\end{equation}
with the initial conditions $f_{0,0}(x)=1$, and $f_{k,i}(x)=0$ for $k<i$.
\end{lemma}

\begin{proof}
By the characterization in Lemma~\ref{char:12-0}, each ascent sequence $e\in\PA_{k+1,i}(\underline{12}0)$ with tail sequence $\tt(e)=t_1t_2\cdots t_{k+2}$ and the penultimate tail being $e_p=j$ ($0\le j\le i$) is  decomposed into
\begin{itemize}
\item the prefix $e_1e_2\ldots e_p\in\PA_{k,j}(\underline{12}0)$,
\item the entries $e_{p+1}>e_{p+2}>\ldots>e_{\len(e)-1}$ forming a subset of the interval $[i+1,k+1]$ with the restriction that such a subset must be non-empty whenever $j=i$ (since $e$ is primitive), and
\item the last entry $e_{\len(e)}=i$.
\end{itemize}
 Now if we take the weight into consideration, recursion~\eqref{rec:12-0} follows from the decomposition above immediately.
\end{proof}

We are now ready to prove the following expression for $f_{k,i}(x)$.

\begin{theorem}\label{gen:conj1}
For $0\le i\le k$,  we have
\begin{align}\label{f_ki}
f_{k,i}(x)&=(1+x)^{\binom{i}{2}}\prod_{\ell=i+1}^{k}\left((1+x)^{\ell}-1\right).
\end{align}
\end{theorem}

Conjecture~\ref{asc:12-0} is an immediate consequence of Theorem~\ref{gen:conj1}.
\begin{proof}[{\bf Proof of Conjecture~\ref{asc:12-0}}]
It follows from recursion~\eqref{rec:12-0} and formula~\eqref{f_ki} that
$$
\sum_{i=0}^k f_{k,i}(x)=f_{k+1,k+1}(x)=(1+x)^{\binom{k+1}{2}},
$$
which establishes~\eqref{sum over PA} and thus Conjecture~\ref{asc:12-0} is true.
\end{proof}

We are going to prove Theorem~\ref{gen:conj1} by induction based on recursion~\eqref{rec:12-0}.

\begin{proof}[{\bf Proof of Theorem~\ref{gen:conj1}}]
We will prove the result by induction on $k$.
 The first few values 
$$f_{0,0}(x)=1, \; f_{1,0}(x)=x \text{ and } f_{1,1}(x)=1$$ can be readily checked. Suppose that~\eqref{f_ki}  holds for all $k\le m$ and $0\le i\le k$, for certain integer $m\ge 1$. We compute the case with $k=m+1$.

By recursion~\eqref{rec:12-0}, we have
\begin{align*}
f_{m+1,m+1}(x)=\sum_{i=0}^{m}f_{m,i}(x)&=\sum_{i=0}^{m}(1+x)^{\binom{i}{2}}\prod_{\ell=i+1}^{m}\left((1+x)^{\ell}-1\right)\\
&=(1+x)^{\binom{m}{2}}+\left((1+x)^{m}-1\right)\sum_{i=0}^{m-1}(1+x)^{\binom{i}{2}}\prod_{\ell=i+1}^{m-1}\left((1+x)^{\ell}-1\right)\\
&=(1+x)^{\binom{m}{2}}+\left((1+x)^{m}-1\right)\sum_{i=0}^{m-1}f_{m-1,i}(x)\\
&=(1+x)^{\binom{m}{2}}+\left((1+x)^{m}-1\right)(1+x)^{\binom{m}{2}}\\
&=(1+x)^{\binom{m+1}{2}}.
\end{align*}
For $0\leq i\leq m$, it follows from 
$$\sum_{i=0}^{k}f_{k,i}(x)=f_{k+1,k+1}(x)=(1+x)^{\binom{k+1}{2}}\quad\text{($1\leq k\leq m$)}
$$ 
 and  recursion~\eqref{rec:12-0} that
\begin{align*}
f_{m+1,i}(x)&=\sum_{j=0}^{i}(1+x)^{m+1-i}f_{m,j}(x)-f_{m,i}(x)\\
&=(1+x)^{m+1-i}\sum_{j=0}^{i}(1+x)^{\binom{j}{2}}\prod_{\ell=j+1}^{m}\left((1+x)^{\ell}-1\right)-(1+x)^{\binom{i}{2}}\prod_{\ell=i+1}^{m}\left((1+x)^{\ell}-1\right)\\
&=\left((1+x)^{m+1-\binom{i+1}{2}}\sum_{\ell=0}^{i}f_{i,\ell}(x)-1\right)(1+x)^{\binom{i}{2}}\prod_{\ell=i+1}^{m}\left((1+x)^{\ell}-1\right)\\
&=\left((1+x)^{m+1}-1\right)(1+x)^{\binom{i}{2}}\prod_{\ell=i+1}^{m}\left((1+x)^{\ell}-1\right)\\
&=(1+x)^{\binom{i}{2}}\prod_{\ell=i+1}^{m+1}\left((1+x)^{\ell}-1\right).
\end{align*}
Thus, we have  verified  the case with $k=m+1$ for~\eqref{f_ki}. The proof is now completed by induction.
\end{proof}

\section{On $\underline{12}0$-avoiding inversion  sequences}
\label{sec:3}

\subsection{Proof of Conjecture~\ref{inv:12-0}}
In this subsection, we develop a delicate algorithm to construct recursively $\underline{12}0$-avoiding inversion  sequences, which leads to a proof of Conjecture~\ref{inv:12-0}. 

The following operations are quite standard (cf. \cite{cor}) for constructing new inversion sequences from old ones.
For an inversion sequence $e=e_1e_2\ldots e_n\in\I_n$ and any integer $t$, let $$\sigma_t(e):=e_1'e_2'\ldots e_n', \text{ where }e_i'=\begin{cases}0 & \text{if $e_i=0$},\\
e_i+t & \text{otherwise}.\end{cases}$$
Note that the image $\sigma_t(e)$ is not necessarily an inversion sequence. And sometimes we need to apply $\sigma_t$ on substrings of an inversion sequence.
We use concatenation to add an entry to the beginning or the end of an inversion sequence: $0\cdot e$ is the inversion sequence $0e_1e_2\ldots e_n$ and for $0\le i\le n$, $e\cdot i$ is the inversion sequence $e_1e_2\ldots e_n i$. For any sequence $s$, not necessarily an inversion sequence, we use $\min(s)$ to denote the value of the smallest entry in $s$.

Quite recently, Beaton, Bouvel, Guerrini and Rinaldi~\cite[Prop.~19]{bbgr} rephrased \eqref{recurrence} in terms of the following succession rule, and reproved the statement of Corteel et al. for $\I_n(110)$ by explaining their growth subjected to this rule. 
\begin{align}\label{rule}
\Omega_{\mathrm{pCat}}=\begin{cases}
(1) \\
(k) \leadsto (1), (2)^2, (3)^3,\ldots,(k)^k,(k+1).
\end{cases}
\end{align}
Here $(i)^i$ means $i$ copies of $(i)$.
The {\em powered Catalan generating tree} (actually an infinite rooted tree) can be constructed from $\Omega_{\mathrm{pCat}}$ like this: the root is $(1)$ and the children of a vertex labelled $(k)$ are those generated according to the rule $\Omega_{\mathrm{pCat}}$. Note that the number of vertices  at level $n$ that carry the label $(k)$  in the powered Catalan generating tree is precisely the quantity $c_{n,k}$.  

Our strategy to prove Conjecture~\ref{inv:12-0} is to show that the family $\{\I_n(\underline{12}0)\}_{n\ge 1}$ also obeys the succession rule $\Omega_{\mathrm{pCat}}$. We remark that the first step is the same as given in \cite{bbgr}, while the second step involving ``Algorithm BS'' is substantially different and crucial in dealing with $\underline{12}0$-avoiding, rather than $110$-avoiding inversion sequences.

\begin{proof}[{\bf Proof of Conjecture~\ref{inv:12-0}}]
For $1\leq k\leq n$, let $\I_{n,k}(\underline{12}0):=\{e\in\I_n(\underline{12}0):\zero(e)=k\}$.
Let $e=e_1\ldots e_n\in\I_{n,k}(\underline{12}0)$ and suppose its $k$ zero entries are indexed as $e_{i_1}(=e_1),e_{i_2},\ldots,e_{i_k}$. Since $e$ is $\underline{12}0$-avoiding, it uniquely decomposes as $$e=0W_10W_20\ldots0W_{k-1}0W_k,$$ where for $1\le j\le k-1$, $W_j$ is a non-increasing, zero-free substring of length $i_{j+1}-i_j-1$, and $W_k$ is a $\underline{12}0$-avoiding, zero-free substring of length $n-i_k$.
\begin{description}
	\item[Step I] Set $e'=0\cdot\sigma_1(e)=0e_1'\ldots e_n'=00\sigma_1(W_1)0\sigma_1(W_2)0\ldots 0\sigma_1(W_{k-1})0\sigma_1(W_k).$
	\item[Step II] Transform $e'$ into one or more $\underline{12}0$-avoiding inversion sequences, according to the following three succession cases.
	\begin{description}
		\item[$\mathbf{(k+1)}$] Set $e^{(k+1)}=e'$.
		\item[$\mathbf{(1)}$] Replace each of $e_{i_1}',e_{i_2}',\ldots,e_{i_k}'$ by $1$, and denote the new sequence by $e^{(1)}$.
		\item[$\mathbf{(j)}$] For any $2\le j\le k$, replace each of $e_{i_{j+1}}',e_{i_{j+2}}',\ldots,e_{i_k}'$ by $1$. Choose one integer $1\le m\le j$, then go on to replace the zero $e_{i_m}'$ by $1$, and denote this new sequence by $e''$. Apply the following Algorithm BS on $e''$. The output sequence is denoted as $e^{(j,m)}$.
	\end{description}
\end{description}
\begin{framed}
\begin{center}
 {\bf Algorithm BS} (backward shift)
\end{center}
\noindent Input sequence $$e=00U_10\ldots 0U_{m-1}1U_{m}0\ldots0U_j1U_{j+1}1\ldots1U_k,$$ where the substrings $U_1,\ldots,U_k$ contain neither $0$ nor $1$, and the $1$ between $U_{m-1}$ and $U_m$ is the only $1$ to the left of $U_j$.

If $m=j$ or $U_m=\emptyset$, output $e$ as is. Note that in this case $e$ contains at most one $1$ between the $0$s.

Otherwise, initiate $R=U_m$ and we go through the following steps to locally transform certain substring of $e$. 
\begin{itemize}
	\item[Step 1] Find the substring $L\delta R$, where $\delta=0$ or $1$, and $L$ is the maximal zero-free substring extended to the left of $\delta$.
	\item[Step 2] Transform $L\delta R\rightarrow L\cdot\sigma_{-1}(R)\cdot\delta$. 
    \item[Step 3] If $L\neq\emptyset$, put $R=L$ and go back to Step 1.
 
    Else if $\min(\sigma_{-1}(R))=1$, terminates.

    \quad Else put $R=\sigma_{-1}(R)$ and go back to Step 1.
\end{itemize}
\noindent Output the final sequence.

\begin{example} Take $e=011100630870020\in\I_{15,7}(\underline{12}0)$ with $j=5$ and $m=3$ for example. Applying the succession rules $\mathbf{(j)}$ and the algorithm BS gives
\begin{align*}
e'=0022200740980030\rightarrow e''&=0022201\red{74}0980131\rightarrow 002220\red{63}10980131\\
&\quad\rightarrow 00\red{222}52010980131\rightarrow 0111052010980131=e^{(5,3)},
\end{align*}
where all the involved substrings $R$ are colored in red.
\end{example}
\end{framed}
For well-definedness, one checks that following the succession rules $\mathbf{(k+1)}$, $\mathbf{(1)}$ and $\mathbf{(j)}$, we end up respectively, with one sequence $e^{(k+1)}\in\I_{n+1,k+1}(\underline{12}0)$, one sequence $e^{(1)}\in\I_{n+1,1}(\underline{12}0)$, and $j$ sequences $e^{(j,m)}\in\I_{n+1,j}(\underline{12}0)$ for $1\le m\le j$ and $2\le j\le k$.

To complete the proof, we have to show that if we apply the above process for each sequence in $\I_{n}(\underline{12}0)$, we generate every sequence in $\I_{n+1}(\underline{12}0)$ precisely once. The first thing to notice is that $e^{(k+1)}$ contains no $1$s, $e^{(1)}$ has only one $0$, and $e^{(m)}$ has at least two $0$s and at least one $1$. So these three cases are mutually exclusive. It should be clear how to invert $e^{(k+1)}$ or $e^{(1)}$ to recover $e$, so it suffices to invert $e^{(j,m)}$. This is done by first applying the forward shift algorithm below to $e^{(j,m)}$, which outputs the sequence $e''$; then obtaining  $e'$ from $e''$ by replacing all $1$s by $0$s; and finally  deriving $e$ from $\sigma_{-1}(e')=0\cdot e$.
\begin{framed}
\begin{center}
 {\bf Algorithm FS} (forward shift)
\end{center}
\noindent Input sequence $e$, which has at least two $0$s and at least one $1$. We call the substring inbetween the leftmost $0$ and the rightmost $0$ the {\em zero zone} of $e$.

If $e$ has less than two $1$s in the zero zone, output $e$ as is.

Otherwise we can write $$e=0\ldots0V_i 10V_{i+1}0\ldots 0V_j10\ldots0V_{l},$$ where the substrings $V_1,\ldots,V_{j},V_{j+1},\ldots,V_{l-1}$ are zero-free and non-increasing, $V_{l}$ is zero-free and $\underline{12}0$-avoiding, and  $V_i1$  (resp.~$V_j1$) contains the leftmost (resp. rightmost) $1$ in the zero zone. Now initiate $L=V_i1$ and we go through the following steps to locally transform certain substring of $e$. 
\begin{itemize}
	\item[Step 1] Find the substring $L 0 R$, where $R$ is the maximal zero-free substring extended to the right of $0$.
	\item[Step 2] Transform $L 0 R\rightarrow 0\cdot\sigma_{1}(L)\cdot R$. 
    \item[Step 3] If $R=R'1$ ends with the rightmost $1$ in the zero zone, continue.

    If $R'=\emptyset$, transform $\sigma_1(L)10\rightarrow 1\cdot\sigma_{2}(L)\cdot0$ and terminates.

    Else transform $R'10\rightarrow 1\cdot\sigma_{1}(R')\cdot0$ and terminates.

    \noindent Else if $R=\emptyset$, put $L=\sigma_1(L)$ and go back to Step 1.

    \noindent Else put $L=R$ and go back to Step 1.
\end{itemize}
\noindent Output the final sequence.

\begin{example} Take $e=0111052010980131\in\I_{16,5}(\underline{12}0)$  for example. Applying  the algorithm FS gives
\begin{align*}
e\rightarrow 0022252010980131\rightarrow 0022206310980131\rightarrow0022201740980131=e''.
\end{align*}
\end{example}
\end{framed}

In conclusion, we have proved that $\underline{12}0$-avoiding inversion sequences grow according to  the rule $\Omega_{\mathrm{pCat}}$ if every sequence with $k$ zeros is represented by $(k)$. This completes the proof of the conjecture. 
\end{proof}

\begin{example}\label{eg-algorithm}
In this example, we find all $1+1+2+3+4=11$ images of an inversion sequence $e\in\I_{12,4}(\underline{12}0)$, following the steps described in the proof above. 
\begin{align*}
e=010211002565\quad &\leadsto \quad e^{(4+1)}=0020322003676. \\
e=010211002565\quad &\leadsto \quad e^{(1)}=0121322113676. \\
e=010211002565\quad &\leadsto \quad \begin{cases}e^{(2,2)}=0021322113676, \\ e^{(2,1)}=0110322113676.\end{cases}\\
e=010211002565\quad &\leadsto \quad \begin{cases}e^{(3,3)}=0020322113676, \\ e^{(3,2)}=0102111013676, \\ e^{(3,1)}=0110322013676. \\ \end{cases}\\
e=010211002565\quad &\leadsto \quad \begin{cases}e^{(4,4)}=0020322013676, \\ e^{(4,3)}=0020322103676, \\ e^{(4,2)}=0102111003676, \\ e^{(4,1)}=0110322003676. \end{cases}
\end{align*}
\end{example}

\subsection{The last entry statistic} In this subsection, we study the last entry statistic of inversion sequences and obtain a new succession rule for powered Catalan numbers. For an inversion sequence $e\in\I_n$, let $\last(e)=e_n$ be the {\em last entry} of $e$. The last entry statistic has been found to be useful in solving  two enumeration conjectures in~\cite{kl2}. By comparing the construction of the rule $\Omega_{\mathrm{pCat}}$ for $\underline{12}0$-avoiding inversion sequences in the proof of Conj.~\ref{inv:12-0} and that for $110$-avoiding inversion sequences in the proof of~\cite[Prop.~19]{bbgr}, we have the following equidistribution. 
\begin{proposition}
The tripe $(\last,\zero,\rmin)$ has the same distribution over $\I_n(\underline{12}0)$ and $\I_n(110)$, where $\rmin(e)$ is the number of right-to-left minima of an inversion sequence $e$. 
\end{proposition}

Lin and Yan~\cite[Lem.~2.19]{ly} showed that Baril and Vajnovszki's $b$-code~\cite{bv} restricts  to  a bijection between $\S_n(3\underline{21}4)$ and $\I_n(\underline{12}0)$. 
 For a permutation $\pi=\pi_1\pi_2\cdots\pi_n\in\S_n$, define the encoding $\Theta$ by
$$
\Theta(\pi)=(e_1,e_2,\ldots,e_n),\quad\text{where $e_i:=\left|\{j: \text{$j<i$ and $\pi_j>\pi_i$}\}\right|$}.
$$
The encoding $\Theta$,  known as {\em invcode} of permutations, is a variation of the famous Lehmer code. 
One interesting feature of $\Theta$ that the $b$-code does not possess is
$\last(\pi)=\last(\Theta(\pi))$,
where $\last(\pi):=n-\pi_n$.

\begin{proposition}
The invcode $\Theta$ restricts to a bijection between $\S_n(3\underline{21}4)$ and $\I_n(\underline{12}0)$. 
Consequently, the triple  $(\last,\zero,\rmin)$ over $\I_n(\underline{12}0)$ (or  $\I_n(110)$) is equidistributed with $(\last,\lmax,\rmax)$ over $\S_n(3\underline{21}4)$, where $\lmax(\pi)$ (resp.~$\rmax(\pi)$) denotes  the number of left-to-right maxima (resp.~right-to-left maxima) of a permutation $\pi$.
\end{proposition}
\begin{proof}
Let $\pi\in\S_n$ and $e=\Theta(\pi)$. If $\pi$ contains the pattern $3\underline{21}4$, then there exists $1\leq i<j<k-1\leq n-1$ such that  $\pi_k>\pi_i>\pi_j>\pi_{j+1}$ and $\pi_{\ell}<\pi_i$ for each $j+1<\ell<k$. 
Thus, we have  $e_k<e_j<e_{j+1}$ and so $e_je_{j+1}e_k$ forms a $\underline{12}0$ pattern in $e$. Conversely, suppose that $1\leq i<j-1$ and $e_ie_{i+1}e_j$ is a $\underline{12}0$ pattern of $e$. Since $e_i>e_j$, we have $\pi_i<\pi_j$ and there exists $1\leq k<i$ such that $\pi_i<\pi_k<\pi_j$. Now $\pi_k\pi_i\pi_{i+1}\pi_j$ forms a pattern $3\underline{21}4$ in $\pi$. This completes the proof.
\end{proof}

In~\cite[Prop.~25]{bbgr}, Beaton, Bouvel, Guerrini and Rinaldi obtained another succession rule for the powered Catalan numbers, which is essentially different from $\Omega_{\mathrm{pCat}}$:
\begin{align*}
\Omega_{\mathrm{1\underline{23}4}}=\begin{cases}
(1,1) \\
(1,q) \leadsto (1,q+1), (2,q),\ldots,(1+q,1),\\
(p,q)\leadsto (1,p+q), (2,p+q-1),\ldots,(p,q+1),\\
\qquad\qquad (p+1,0),\dots,(p+q,0),\qquad\qquad\qquad\,\,\,\,\quad\text{if $p>1$}.
\end{cases}
\end{align*} 
The consideration of the last entry statistic on $\underline{12}0$-avoiding inversion sequences leads to a third succession rule for the powered Catalan numbers. 

For a sequence $e\in\I_n(\underline{12}0)$, let us introduce the parameters $(p,q)$ of $e$ by 
$$
p:=|\{k: (e_1,e_2,\ldots,e_n,k)\in\I_{n+1}(\underline{12}0)\text{ and } k>e_n\}|=n-e_n
$$
and 
$$
q:=|\{k: (e_1,e_2,\ldots,e_n,k)\in\I_{n+1}(\underline{12}0) \text{ and } k\leq e_n\}|.
$$

\begin{proposition}\label{succ:3}
The $\I_n(\underline{12}0)$-avoiding inversion sequences grow according to the following succession rule 
\begin{align*}
\Omega_{\mathrm{\underline{12}0}}=\begin{cases}
(1,1) \\
(p,q)\leadsto (p,2), (p-1,3),\ldots,(1,p+1),\\
\qquad\qquad (p+1,q),(p+2,q-1), \dots, (p+q,1).
\end{cases}
\end{align*}
\end{proposition}
\begin{proof}
Let $e$ be a sequence in $\I_n(\underline{12}0)$ with parameters $(p,q)$. It is clear that the sequence $s:=(e_1,e_2,\ldots,e_n,k)$ is in $\I_{n+1}(\underline{12}0)$ if and only if $n\geq k\geq e_n-q+1$, where $e_n-q+1$ equals the largest ascent bottom of $e$. We consider two cases:
\begin{itemize}
\item If $e_n<k\leq n$, then $e_n k$ forms an ascent of $f$ whose ascent bottom $e_n$ is obviously not smaller than $e_n-q+1$. So if we write $k=e_n+i$ for some $1\leq i\leq n-e_n=p$, then the parameters of $f$ are $(p-i+1,i+1)$. 
\item If $e_n-q+1\leq k\leq e_n$, then $k=e_n+1-i$ for some $1\leq i\leq q$. In this case, the parameters of $f$ are $(p+i,q+1-i)$. 
\end{itemize}
Summing over all the above two cases results in the succession rule $\Omega_{\mathrm{\underline{12}0}}$.
\end{proof}

\section{Two equidistribution conjectures}

The classification of Wilf equivalences for vincular patterns of length $3$ in inversion sequences has been completed, thanks to Auli and Elizalde's recent work\footnote{Auli and Elizalde independently initiated their work, we thank them for keeping us informed.}~\cite{auli3}.
Towards the complete classification of vincular patterns of length $4$ in permutations, 
 Baxter and Shattuck conjectured~\cite{bs} that $\S_n(\underline{23}14)$ has cardinality $p_n$, the $n$-th powered Catalan number. In their attempt to prove this conjecture, Beaton, Bouvel, Guerrini and Rinaldi~\cite[Conj.~23]{bbgr} found the following refinement.
\begin{conjecture}\label{23-1-4}
The number of permutations of $\S_n(\underline{23}14)$ with $k$ right-to-left minima is $c_{n,k}$. 
\end{conjecture}

Conjecture~\ref{23-1-4} is equivalent to the assertion that the statistic `$\zero$' over $\I_n(\underline{12}0)$ or $\I_n(110)$ has the same distribution as `$\rmin$' over $\S_n(\underline{23}14)$, where $\rmin(\pi)$ denotes the number of {\em right-to-left minima} of a permutation $\pi$. 
Using Maple program, we find the following refinement of Conjecture~\ref{23-1-4}. 

\begin{conjecture}\label{conj:fu-lin}
 The quadruple $(\rmin,\lmin,\rmax,\asc)$ on $\S_n(\underline{23}14)$ has the same distribution as $(\zero,\maxi,\rmin,\rep)$ on $\I_n(110)$. 
\end{conjecture}
Here  we use $\lmin(\pi)$ (resp.~$\asc(\pi)$) to denote  the number of {\em left-to-right minima} (resp.~{\em ascents}) of a permutation $\pi$. And for an inversion sequence $e\in\I_n$, the two statistics  involved are
\begin{align*}
\maxi(e):=|\{i\in[n]: e_i=i-1\}| \quad\text{and}\quad\rep(e):=n-|\{e_1,e_2,\ldots,e_n\}|.
\end{align*}
Conjecture~\ref{conj:fu-lin} has been verified for $1\leq n\leq9$.

Finally,  the consideration of the last entry statistic  leads to another refinement of Baxter and Shattuck's enumeration conjecture. 
\begin{conjecture}\label{conj:fu-lin2}
The pair $(\last,\rmax)$ on $\S_n(\underline{23}14)$ has the same distribution as $(\last,\rmin)$ on $\I_n(\underline{12}0)$. 
\end{conjecture}
Conjecture~\ref{conj:fu-lin2} has also been verified for $1\leq n\leq9$. In view of Proposition~
\ref{succ:3}, it would be interesting to show that $\underline{23}14$-avoiding permutations grow according to the rule $\Omega_{\mathrm{\underline{12}0}}$. One remarkable special case of Conjecture~\ref{conj:fu-lin2}  is that 
$$
|\{\pi\in\S_n(\underline{23}14): \pi_n=n\}|=B_{n-1}=|\{e\in\I_n(\underline{12}0): e_n=0\}|, 
$$
which follows from the enumeration results in~\cite{cla,ly}. 
Here $B_{n}$ is the $n$-th {\em Bell number}.

\section*{Acknowledgement}
This work was initiated in the summer of 2019, when both authors were visiting the Research Center for Mathematics and Interdisciplinary Sciences at Shandong University. They would like to thank the center for the excellent working condition and the hospitality extended during their stay. The first named author was supported by the National Science Foundation of China grant 11871247 and the project of Qilu Young Scholars of Shandong University.


\begin{thebibliography}{99}


\bibitem{auli} J.S. Auli and S. Elizalde, Consecutive patterns in inversion sequences, Discrete Math. Theor. Comput. Sci., {\bf21} (2019),  $\#6$..

\bibitem{auli2}J.S. Auli and S. Elizalde, Consecutive patterns in inversion sequences II: avoiding patterns of relations,  J. Integer Seq., {\bf22} (2019), Article 19.7.5. 

\bibitem{auli3}J.S. Auli and S. Elizalde, Wilf equivalences between vincular patterns in inversion sequences, \href{https://arxiv.org/abs/2003.11533}{arXiv:2003.11533}

\bibitem{bas} E. Babson and E. Steingr\'imsson, Generalized permutation patterns and a classification of the Mahonian statistics, S\'emin. Lotharingien Comb., {\bf B44b} (2000), 18pp.

\bibitem{bv} J.~L.~Baril and V.~Vajnovszki, A permutation code preserving a double Eulerian bistatistic, Discrete Appl. Math., {\bf 224} (2017), 9--15.

\bibitem{bp} A.M. Baxter and L.K. Pudwell, Ascent sequences avoiding pairs of patterns,  Electron. J. Combin., {\bf 22(1)} (2015), \#P1.58.


\bibitem{bs} A.M. Baxter and M. Shattuck, Some Wilf-equivalences for vincular patterns, J. Comb., {\bf6} (2015), 19--45.

\bibitem{bbgr} N.R. Beaton, M. Bouvel, V. Guerrini and S. Rinaldi, Enumerating five families of pattern-avoiding inversion sequences; and introducing the powered Catalan numbers, Theoret. Comput. Sci., {\bf777} (2019), 69--92.

\bibitem{bcdk}M. Bousquet-M\'elou, A. Claesson, M. Dukes and S. Kitaev,
$({\bf2+2})$-free posets, ascent sequences and pattern avoiding permutations,
J. Combin. Theory Ser. A, {\bf117} (2010), 884--909.

\bibitem{bgrr} M. Bouvel, V. Guerrini, A. Rechnitzer, and S. Rinaldi, Semi-Baxter and strong-Baxter: two relatives of the Baxter sequence,  SIAM J. Discrete Math., {\bf 32(4)} (2018), 2795--2819.

\bibitem{cjl} W. Cao, E.Y. Jin and Z. Lin, Enumeration of inversion sequences avoiding triples of relations, Discrete Appl. Math., {\bf260} (2019), 86--97.

\bibitem{cddds} W.Y.C. Chen, A.Y.L. Dai, T. Dokos, T. Dwyer and B.E. Sagan, On $021$-avoiding ascent sequences,  Electron. J. Combin., {\bf 20(1)} (2013), \#P76.

\bibitem{cla}A. Claesson, Generalized Pattern Avoidance, European J. Combin., 22 (2011), 961--971.

\bibitem{cor} S. Corteel, M. Martinez, C.D. Savage and M. Weselcouch, Patterns in Inversion Sequences I,  Discrete Math. Theor. Comput. Sci., {\bf18} (2016),  $\#2$.

\bibitem{ds} P. Duncan and E. Steingr\'imsson, Pattern avoidance in ascent sequences, Electron. J. Combin., {\bf 18(1)} (2011), \#P226.

\bibitem{eklp} N. Elkies, G. Kuperberg, M. Larsen and J. Propp, Alternating-sign matrices and domino tillings. I., J. Algebraic Combin., {\bf1} (1992), 111--132.

\bibitem{fjlyz} S. Fu, E.Y. Jin, Z. Lin, S.H.F. Yan and R.D.P. Zhou, A new decomposition of ascent sequences and Euler--Stirling statistics,  J. Combin. Theory Ser. A, {\bf170} (2020), Article 105141.

\bibitem{ki} S. Kitaev, {\em Patterns in permutations and words}, Springer Science \& Business Media, 2011.

\bibitem{kl2} D. Kim and Z. Lin, Refined restricted inversion sequences (extended abstract at FPSAC2017), S\'em. Lothar. Combin., {\bf78B} (2017), Art. 52, 12 pp.


\bibitem{kl} Z. Lin and D. Kim, A sextuple equidistribution arising in Pattern Avoidance, J. Combin. Theory Ser. A, {\bf155} (2018), 267--286.

\bibitem{lin}Z. Lin, Restricted inversion sequences and enhanced $3$-noncrossing partitions, European J. Combin., {\bf70} (2018), 202--211.

\bibitem{lin2} Z. Lin, Patterns of relation triples in inversion and ascent sequences, Theoret. Comput. Sci., {\bf804} (2020), 115--125..

\bibitem{ly} Z. Lin and S.H.F. Yan, Vincular patterns in inversion sequences, App. Math. Comput., {\bf364} (2020), Article 124672.

\bibitem{mash}T. Mansour and M. Shattuck,
Pattern avoidance in inversion sequences,  Pure Math. Appl. (PU.M.A.), {\bf25} (2015), 157--176.

\bibitem{ms} M. Martinez and C.D. Savage, Patterns in Inversion Sequences II: Inversion Sequences Avoiding Triples of Relations,  J. Integer Seq., {\bf21} (2018), Article 18.2.2.

\bibitem{oeis} OEIS Foundation Inc., The On-Line Encyclopedia of Integer Sequences,  \href{http://oeis.org}{http://oeis.org}, 2019.

\bibitem{sav} C.D. Savage, The Mathematics of lecture hall partitions,  J. Combin. Theory Ser. A, {\bf144} (2016), 443--475.


\bibitem{yan0} S.H.F. Yan, Ascent sequences and $3$-nonnesting set partitions, European J. Combin., {\bf39} (2014), 80--94.

\bibitem{yan} S.H.F. Yan, Bijections for inversion sequences, ascent sequences and $3$-nonnesting set partitions, App. Math. Comput., {\bf325} (2018), 24--30.


\end{thebibliography}
\end{document}